\documentclass{amsart}

\newtheorem{theorem}{Theorem}[section]
\newtheorem{lemma}[theorem]{Lemma}
\newtheorem{proposition}[theorem]{Proposition}
\newtheorem{corollary}[theorem]{Corollary}

\theoremstyle{definition}
\newtheorem{definition}[theorem]{Definition}
\newtheorem{example}[theorem]{Example}

\usepackage{amscd,amssymb}

\begin{document}

\title[$S$-Prime Right Submodules and an $S$-Version of Prime Avoidance ]
 {$S$-prime right submodules and an $S$-version of prime avoidance}
\author[Alaa Abouhalaka]{Alaa Abouhalaka}

\address{Department of Mathematics,
\c{C}ukurova University, 01330 Balcal\i,
 Adana, Turkey}
\email{alaa1aclids@gmail.com}

\thanks
{}

\

%
\date{\today}
\dedicatory{Last Revised:\\ \today}
\begin{abstract}
Let $S$ be an $m$-system of a ring $R$, and $P$ a submodule of a right $R$-module $M$. This paper, presents the notion of  $S$-prime submodule  and provides some properties and equivalent definitions.  We define $S$-multiplication right module, and prove that in multiplication ($S$-multiplication) right $R$-module $M$, the ideal $(P :_R M)$ is a right $S$-prime ideal of $R$ if and only if $P$ is an $S$-prime submodule of $M$. Moreover,  we give an $S$-version of prime avoidance lemma. Furthermore, we  define $S$-finite and $S$-Noetherian  right modules  following the definitions in \cite{A}. We prove  that a multiplication finitely generated right $R$-module $M$ is  $S$-Noetherian if  $(N:_RM)$  is an $S$-prime ideal of $R$, for all submodules $N$ of $M$. In addition, we give  some examples of right $S$-Noetherian rings. 
\end{abstract}
\label{page:firstblob}
\maketitle

\section{Introduction} Let $S$ be a multiplicatively closed set of a commutative ring $R$.

Anderson and Dumitrescu, in \cite{DD}, introduced the notion of $S$-finite ideals and $S$-Noetherian rings. An ideal $I$ is called $S$-finite if there exists a finitely generated  ideal $J$ of $R$ such that $Is\subseteq J\subseteq I$ for some $s\in S$. If every ideal in $R$ is $S$-finite, then $R$ is termed an $S$-Noetherian ring. Hamed and Malek, in \cite{HK}, introduced the notion of $S$-prime ideals in commutative rings. An ideal $P$ of $R$, disjoint from $S$, is $S$-prime if there exists an element $s \in S$ such that for all $a, b \in R$, if $ab \in P$, then $sa \in P$ or $sb \in P$. They demonstrated that a commutative ring $R$ is $S$-Noetherian if and only if every $S$-prime ideal is $S$-finite. Several studies in commutative rings have generalized the concept of $S$-prime ideals, as evidenced by papers such as \cite{EM} and \cite{SV}. Valuable results in commutative $S$-Noetherian rings, including practical applications in $S$-principal ideal rings, can be found in works like \cite{D} and \cite{HOW}. Furthermore, the concept of $S$-prime submodules was introduced in \cite{TOK}, leading to subsequent studies and generalizations in papers such as \cite{DDD} and \cite{HTS}. Refer to \cite{TOK}, a  submodule $P$  of a $R$-module $M$ with $(P :_R M) \cap S=\phi$ is called  an $S$-prime submodule if for $a \in R$ and $m\in M$, whenever $ma\in P$ implies    either $ as\in(P :_R M)$ or $m s\in  P$ for some $s\in S$. An noteworthy finding is presented in \cite{DDD}, wherein an $S$-version of the prime avoidance lemma is provided specifically for the commutative case.

The extension of mathematical concepts to noncommutative cases has garnered significant interest among researchers. In particular, those topics related to prime
ideals, as examples \cite{AAA}, \cite{A3},  \cite{MS}, and \cite{ND}. Furthermore, in \cite{MR} we see generalization of Cohen and Kaplansky Theorem to noncommutative rings. 

Recently, several authors have initiated research on  right $S$-Noetherian rings. For instance, in reference like \cite{JGJ}, Hilbert basis Theorem for right $S$-Noetherian rings is given. Of particular note is the $S$-version of the Cohen’s Theorem in \cite{RT} and  $S$-version of the Eakin-Nagata-Eisenbud Theorem  in \cite{LEE}. In addition, the $S$-version of the Eakin-Nagata-Eisenbud Theorem and 
$S$-version of Cohen’s Theorem for $S$-principal right ideal ring are given in \cite{JB}. In addition to generalizations that include right modules as we see in \cite{J} a  study of $S$-injective right modules.

Motivated   by the fact that prime ideals are closely tied  to the $m$-system concept in noncommutative rings. The author, in \cite{A},  introduced (as a new generalization into noncommutative rings) the concept of right $S$-prime (right) ideals, $S$-finite (right) ideals, and right $S$-Noetherian rings by considering $S$ as an $m$-system. The author gave an $S$-version of Cohen’s Theorem as follows:

\begin{theorem}  Let $S$ be an $m$-system of a  ring $R$ with identity. If $I\subset RI$ for all right ideals $I$ of $R$, then the following are equivalent:

$(1)$ Every right $S$-prime right ideal of $R$ is $S$-finite.

$(2)$ Every prime right ideal (disjoint from $S$) of $R$ is $S$-finite.

$(3)$ $R$ is a right $S$-Noetherian ring. 
\end{theorem}

In this paper, we continue the work in \cite{A}, and present the concept of $S$-prime submodules of a right $R$-module  $M$ as generalization of the concept of $S$-prime submodules in commutative rings which were introduced in \cite{TOK}. In addition, we show some properties, and equivalent definitions of $S$-prime submodules, we prove that a  submodule $P$ of a right $R$-module $M$ with $(P :_R M) \cap S=\phi$ is $S$-prime if and only if $(P :_M \langle s\rangle)$ is a prime submodule, for some $s\in S$. In addition, we give an $S$-version of prime avoidance lemma by defining the concept of $S$-multiplication right module. 
We also define the concept of  $S$-Finite  and $S$-Noetherian right modules following the definitons in \cite{A}, and we prove that for  a multiplication finitely generated right $R$-module $M$, if $(N:_RM)$  is an $S$-prime ideal of $R$, then $M$ is an $S$-Noetherian right $R$-module. Moreover, we show some examples of right $S$-Noetherian ring.

Throughout the paper, we consider rings that are associative, and with unity. Furthermore, when referring  to an ideal, we specifically mean a proper two sided ideal in the context of our study. In addition, all the modules are right modules, and we use the following notations for an ideal $J$ of $R$, and a  submodule $P$ of a right $R$-module $M$:  

$(P :_R M)=\{ r\in R, Mr \subseteq P\}$ and 
$(P :_M J)=\{ m\in M, mJ \subseteq P\}$

\section{$S$-prime submodules}

The concept of prime submodules  was introduced by Dauns in \cite{JDP}. A submodule $P$ of  a right    $R$-module  $M$ is prime  if for $a \in R$ and $m\in M$ with $mRa\subseteq P$, implies    either $a\in(P :_R M)$ or $m \in  P$. In the following we present the definition of $S$-prime submodule. On the other hand, the concept of  right $S$-prime ideal was introduced by the author in \cite{A} as the following.  

\begin{definition}[ Definition 2.2 of \cite{A}] Let $P$ be an ideal of  a ring $R$, and $S$ be an $m$-system of $R$ such that $P\cap S=\phi$. We call $P$  a right $S$-prime ideal associated with an element $s\in S$ (briefly a right $S$-prime ideal), if  for the ideals $A$, $B$  of $R$ with $AB \subseteq P$, either $A\langle s\rangle \subseteq P$ or $B\langle s\rangle\subseteq P$.  

If the ring $R$ posses an identity, then, Proposition 2.7 of \cite{A}, shows that the ideal $P$ is a right $S$-prime ideal,   if and only if whenever $a, b\in R$ with $aRb\subseteq P$ implies $a\langle s\rangle\subseteq P$ or $b\langle s\rangle\subseteq P$ for some $s\in S$.
\end{definition}

In the following we give the definition of $S$-prime submodule. 

\begin{definition} Let $S \subseteq R$ be an $m$-system of a ring $R$, and $P$ a submodule of right $R$-module $M$ with $(P :_R M) \cap S=\phi$. Then, $P$ is called an $S$-prime submodule if for $a \in R$ and $m\in M$, whenever $mRa\subseteq P$, implies    either $a\langle s\rangle\subseteq(P :_R M)$ or $m\langle s\rangle \subseteq  P$ for some $s\in S$.
\end{definition}

\begin{example}\label{1} Every prime submodule P of  a right    $R$-module  $M$ , with $(P :_R M) \cap S=\phi$, is an S-prime submodule for each m-system S of R,  and the converse is not true in general. However, if  S consist of units of R, then,   every S-prime submodule is a prime submodule.
\end{example}

\begin{example} Let $R=\mathbb{Z}$,  $M=M_2(\mathbb{Z}_4)$, and $P=\bigg\{\left[ \begin{array}{cc}0& 0 \\ 0 & 0\end{array} \right]\bigg\}$. Then, $P$ is a submodule of the $R$-module $M$. Note that $\left[ \begin{array}{cc}2& 2 \\ 0 & 2\end{array} \right].R .2\subseteq P$, however, $2\not\in(P:_RM)=0$ and $\left[ \begin{array}{cc}2& 2 \\ 0 & 2\end{array} \right]\not\in P$, hence, P is not a prime submodule. On the other hand, consider $S=\{s, s^2, s^4, s^8,...\}$ where $s=2$, then $S$ is an $m$-system.  If $\left[ \begin{array}{cc}a_1& a_2 \\ a_3 & a_4\end{array} \right].R .n\subseteq P$, then for $s_1=4\in S$ we obtain $\left[ \begin{array}{cc}a_1& a_2 \\ a_3 & a_4\end{array} \right]\langle s_1\rangle \subseteq  P$, hence,  $P$ is an $S$-prime submodule (recall that $(P :_R M) \cap S=\phi$).
\end{example}

\begin{theorem}\label{EQ} Let $S \subseteq R$ be an $m$-system of a ring $R$, and $P$ a submodule of $R$-module $M$ with $(P :_R M) \cap S=\phi$. Then, $P$ is an $S$-prime submodule associated with $s\in S$ if and only if whenever $NJ \subseteq P$ implies  
\[
J\langle s\rangle \subseteq (P :_R M) \text{  or  }  N\langle s\rangle  \subseteq P,
\]
 for each ideal J of R and submodule N of M.
\end{theorem}

\begin{proof}  Suppose $P$ is an $S$-prime submodule, and let $NJ \subseteq P$ for some ideal $J$ of $R$ and submodule $N$ of $M$. If $N\langle s\rangle  \not\subseteq P$, then there exists $n\in N$ such that $n\langle s\rangle  \not\subseteq P$. 
Now for all $a\in J$, $nRa \subseteq nRaR \subseteq NJ\subseteq P$, hence, by assumption,  $a\langle s\rangle\subseteq(P :_R M)$, consequently, $J\langle s\rangle \subseteq (P :_R M)$.  

Conversely, Suppose that  $mRa\subseteq P$ for some $a \in R$ and $m\in M$, then  $mRRaR\subseteq P$, hence, by assumption, either  $a\langle s\rangle \subseteq \langle a\rangle\langle s\rangle \subseteq (P :_R M)$ or $m\langle s\rangle \subseteq mR\langle s\rangle \subseteq P$. 
\end{proof}

We will denote the sets of all prime submodules and all $S$-prime submodules by $Spec(M_R)$ and $Spec_S (M_R)$, respectively. 

\begin{theorem}\label{MOD}
Let P be a submodule of $R$-module $M$, with $(P :_R M) \cap S=\phi$. Then, $P \in Spec_S (M_R)$ if and only if $(P :_M \langle s\rangle)\in Spec(M_R)$ for some $s\in S$.
\end{theorem}

\begin{proof}  Suppose $(P :_M \langle s\rangle)$ is a prime submodule, and let  $mRa\subseteq P$ for some $a \in R$ and $m\in M$. Then,  $mRa\subseteq (P :_M \langle s\rangle)$, hence, by assumption, either $m\in(P :_M \langle s\rangle)$ which implies $m\langle s\rangle\subseteq P$. Or $a\in((P :_M \langle s\rangle) :_R M)$, hence, $Ma\subseteq(P :_M \langle s\rangle)$, consequently, $Ma\langle s\rangle\subseteq P$, and thus, $a\langle s\rangle\subseteq (P :_R M)$.

Conversely, suppose that $P$ is an $S$-prime submodule, and let

 $mRa\subseteq (P :_M \langle s\rangle)$ for some $a \in R$ and $m\in M$. Then, $mR\langle a\rangle\langle s\rangle\subseteq P $, thus, either $mR\langle s\rangle\subseteq P $ which implies $m\in mR\subseteq (P :_M \langle s\rangle) $. 
 
 Or $\langle a\rangle\langle s\rangle\langle s\rangle\subseteq (P :_R M)$, hence,  $\langle a\rangle RsRsR\subseteq (P :_R M)$, since $S$ is an $m$-system then there exists $r\in R$ such that $s_1=srs\in S$ thus

  $\langle a\rangle\langle s_1\rangle\subseteq (P :_R M)$. Thus, $a\in \langle a\rangle\subseteq ((P :_M \langle s_1\rangle):_RM)$. 
  
  Now let $x\in   (P :_M \langle s_1\rangle)$, then $xRs_1\subseteq x\langle s_1\rangle\subseteq   P$, since $P$ is an $S$-prime submodule associated with $s\in S$, either  $s_1\langle s\rangle\subseteq (P :_R M)$ or $x\langle s\rangle\subseteq   P$. 
  
  If $s_1\langle s\rangle\subseteq (P :_R M)$, then since $S$ is an $m$-system then there exists $r_1\in R$ such that $s_2=s_1r_1s\in S$ thus $s_2\in (P :_R M)$, contradiction. Thus, $x\langle s\rangle\subseteq   P$ and hence $x\in  (P:_M\langle s\rangle)$ which implies $(P :_M \langle s_1\rangle)\subseteq(P:_M\langle s\rangle)$. Hence, $a\in \langle a\rangle\subseteq ((P :_M \langle s_1\rangle):_RM)\subseteq ((P :_M \langle s\rangle):_RM)$, and thus, $(P :_M \langle s\rangle)$ is a prime submodule. 
\end{proof}

\begin{proposition}\label{M1} Let M be a right R-module and S an m-system of R. If P is an S-prime submodule of M, then $(P :_R M)$ is a right S-prime ideal of R.
\end{proposition}
\begin{proof} Let $aRb\subseteq(P :_R M)$ for $a, b \in R$, then $MaRb\subseteq P$. Thus, for all $m \in M$, we have  $(ma)Rb\subseteq P$, and since  $P$ is an $S$-prime submodule, we obtain either $ma\langle s\rangle\subseteq   P$ or $b\langle s\rangle\subseteq   (P :_R M)$. Assume that $b\langle s\rangle\not\subseteq  (P :_R M)$, then $ma\langle s\rangle\subseteq   P$ for all $m\in M$, hence,  $a\langle s\rangle\subseteq   (P :_R M)$. Thus, by $(4)$ of Theorem  2.11 of \cite{A}, $(P:_R M)$ is a right $S$-prime ideal of $R$.
\end{proof}

In refer to \cite{A4}, a  right $R$-module $M$ is called a multiplication module if 
$P =M(P :_R M)$, 
 for every submodule $P$ of $M$.

\begin{proposition}\label{M2} Let  P be a submodule of a  multiplication right  R-module M, and S an m-system of R. If $(P :_R M)$ is a right S-prime ideal of R, then P is an S-prime submodule of M.
\end{proposition}
\begin{proof}  Let $NJ \subseteq P$ for some ideal $J$ of $R$ and submodule $N$ of $M$. Since $M$ is multiplication module, then $N=M(N :_R M)$, hence, $M(N :_R M)J \subseteq P$, and so $(N :_R M)J \subseteq (P :_R M)$. By assumption, either   $J\langle s\rangle \subseteq (P :_R M)$ or  $(N :_R M)\langle s\rangle \subseteq (P :_R M)$, which implies either   $J\langle s\rangle \subseteq (P :_R M)$ or  $N\langle s\rangle=M(N :_R M)\langle s\rangle \subseteq P $. Thus, by Theorem \ref{EQ}, $P$ is an $S$-prime submodule of $M$.
\end{proof}

\begin{corollary}\label{SMM} Let  P be a submodule of a  multiplication right R-module M, and S an m-system of R.

$(1)$  The ideal $(P :_R M)$ is a right S-prime ideal of R if and only if P is an S-prime submodule of M.

$(2)$ If P is an S-prime submodule of  M, then $Ann(M)\subseteq I$ for some right S-prime ideal I of R.
\end{corollary}
\begin{proof} $(1)$ By Proposition \ref{M1} and Proposition \ref{M2}.

$(2)$  For all $x\in Ann(M)$, $Mx=0\in P$, hence, $x\in (P :_R M)$, and by $(1)$, $(P :_R M)$ is  right $S$-prime ideal. 
\end{proof}

In refer to \cite{DDD}, an  $R$-module $M$ over a commutative ring  $R$ is called an $S$-multiplication module (where $S$ is a multiplicatively closed subset of $R$), if for each submodule $N$ of $M$, there exist $s \in S$ and an ideal $I$ of $R$ such that $sN\subseteq IM \subseteq N$. In the following, we give the definition in noncommutative state.

\begin{definition}  Let $S \subseteq R$ be an $m$-system of a ring $R$, and let $M$ be a right  $R$-module. We call $M$  an $S$-multiplication module,  if for each submodule $N$ of $M$, there exist $s \in S$ and an ideal $I$ of $R$ such that $N\langle s\rangle\subseteq MI \subseteq N$.
\end{definition}

Observe that if $M$  is an $S$-multiplication module,  then by our definiton,  $N\langle s\rangle\subseteq MI \subseteq N$, for some submodule $N$ of $M$,  $s \in S$, and an ideal $I$ of $R$. Thus, $N\langle s\rangle\subseteq MI\subseteq M(N :_R M)\subseteq N$. 

\begin{example} Let $S \subseteq R$ be an $m$-system of a ring $R$.

$(1)$ Every multiplication right R-module is an S-multiplication right module,  and if  S consist of units of R, then, the two concepts are coincide.  

$(2)$ If $Ann(M)\cap S\neq\phi$, then take $s\in Ann(M)\cap S$, we get $\langle s\rangle\subseteq  (N :_R M)$, hence, $N\langle s\rangle\subseteq  M(N :_R M)\subseteq N$. Thus, M is $S$-multiplication.

\end{example}

Next we give an $S$-version of $(1)$ of Corollary \ref{SMM}.

\begin{theorem} Let $S \subseteq R$ be an $m$-system of a ring $R$ and let M be an S-multiplication right R-module, and P a submodule of M. The ideal $(P :_R M)$ is a right S-prime ideal of R if and only if P is an S-prime submodule of M.
\end{theorem}

\begin{proof}  Suppose the ideal $(P :_R M)$ is a right $S$-prime ideal, and let $mRa\subseteq P$, for some $a \in R$ and $m\in M$. Since $M$ is $S$-multiplication, there exsits $s\in S$ such that $mR\langle s\rangle\subseteq M(mR:_RM)\subseteq mR$. Let $x\in(mR:_RM)$, then, $MxRa\subseteq mRa\subseteq P$, hence, $xRa\subseteq (P :_R M)$. By assumption, either $x\langle s_1\rangle\subseteq   (P :_R M)$ or $a\langle s_1\rangle\subseteq   (P :_R M)$. Assume that $a\langle s_1\rangle\not\subseteq  (P :_R M)$, then,   $x\langle s_1\rangle\subseteq   (P :_R M)$, consequently, $M(mR:_RM)\langle s_1\rangle \subseteq  M (P :_R M)\subseteq P$. Thus,    $mR\langle s\rangle\langle s_1\rangle \subseteq P$. Now since  $S$ is an $m$-system,  there exists $r_1\in R$ such that $s_2=sr_1s_1\in S$, hence, $m\langle s_2\rangle\subseteq mR\langle s\rangle\langle s_1\rangle \subseteq P$. Thus, $m\in (P:_M \langle s_2\rangle)$. With a similar discussion to that in Theorem \ref{MOD}, we can conclude that $m\in (P:_M \langle s_2\rangle)\subseteq(P:_M \langle s_1\rangle)$, so $m\langle s_1\rangle\subseteq P$.  Therefore, $P$ is an $S$-prim submodule of $M$.

Conversely, suppose that $P$ is an $S$-prime submodule, then, $(P :_R M)$ is a right $S$-prime ideal by Proposition \ref{M1}.
\end{proof}

In refer to \cite{DDD},  we see a generalization of Prime Avoidance Lemma for $S$-multiplication modules, in commutative state, has been proven. In the following we give an  $S$-version of  Prime Avoidance Lemma for $S$-multiplication right modules. The proof is similar to the proof of Theorem 2 of \cite{DDD}.

\begin{theorem} Let S be an m-system of a ring R, and  M be an S-multiplication right R-module and $P, P_1, \cdots, P_n$  submodules of M, at least $n-2$ of which are S-prime. If $P\subseteq P_1\cup P_2\cup \cdots\cup P_n$, then, $P \langle s\rangle\subseteq P_i$ for some $s\in S$.
\end{theorem}

\begin{proof} The proof is by induction for $n\geq2$. If $n=2$, then, $P\subseteq P_1\cup P_2$, hence, either $P\subseteq P_1$ or $P\subseteq P_2$, so $P\langle s\rangle\subseteq P_i$ for all $s\in S$. 

Let us now suppose that the requirement is true, and suppose $n\geq3$, and $P\subseteq P_1\cup P_2\cup \cdots\cup P_{n+1}$.
Assume $P\not\subseteq \bigcup_{i\neq k}P_i$ for each $k\in\{1, 2, \cdots, n+1\}$. 

If for all $k, t\in\{1, 2, \cdots, n+1\}$ with $k\neq t$, and all $s\in S$, 
\[
(P_k:_R M)\langle s\rangle\not\subseteq(P_t:_R M),
 \]
then, we see that $P \langle s\rangle\subseteq P_i$, for some $s\in S$. Because if $P \langle s\rangle\not\subseteq P_i$, for all $s\in S$, then since  $M$ is an $S$-multiplication, there exists $s\in S$ such that $P\langle s\rangle\subseteq  M(P :_R M)\subseteq P$. Hence, $(P:_R M)\langle s\rangle\not\subseteq(P_i:_R M)$, for all $s\in S$,  [if $(P:_R M)\langle s_1\rangle\subseteq(P_i:_R M)$ for some $s_1\in S$, then for some $s_2\in S$, $P\langle s_2\rangle\subseteq P\langle s\rangle\langle s_1\rangle\subseteq  M(P :_R M)\langle s_1\rangle\subseteq M(P_i:_R M) \subseteq P_i$, contradiction.]. Since $n\geq3$, there exists an $S$-prime submodule say $P_{n+1}$. Due to the closure property of submodules, one can show that $ P\cap(\cap^n_{i=1}P_i)\subseteq P_{n+1}$, exactly as it is done in the claim of Theorem 2 of \cite{DDD}. Thus, 
\[
(P :_R M)(P_1 :_R M)\cdots(P_n :_R M)\subseteq (P :_R M)\cap(P_1 :_R M)\cap\cdots\cap(P_n :_R M),
\]
hence, 
\[
(P :_R M)(P_1 :_R M)\cdots(P_n :_R M)\subseteq (P\cap(\cap^n_{i=1}P_i) :_R M)\subseteq (P_{n+1}:_R M)
\]
Since $P_{n+1}$ is an $S$-prime submodule,  $(P_{n+1}:_R M)$ is an $S$-prime ideal by Proposition \ref{M1}. Hence, by Corollary 2.13 of \cite{A}, there exists $s\in S$ such that either $(P :_R M)\langle s\rangle\subseteq(P_{n+1}:_R M)$ or $(P_i :_R M)\langle s\rangle\subseteq(P_{n+1}:_R M)$ for $i\in\{1, 2,\cdots, n\}$, contradiction. 

If, for each pair $(t, k)$  with $k\neq t$, and $k, t\in\{1, 2, \cdots, n+1\}$ there exists  $s\in S$, such that
\[
(P_k:_R M)\langle s\rangle\subseteq(P_t:_R M).
 \]

Then, because $M$ is an $S$-multiplication, there exists $s_1\in S$, such that $P_k\langle s_1\rangle\subseteq  M(P_k :_R M) \subseteq P_k$, thus,

\[
 P_k\langle s_2\rangle\subseteq P_k\langle s_1\rangle\langle s\rangle\subseteq  M(P_k :_R M)\langle s\rangle\subseteq M(P_t:_R M) \subseteq P_t,
 \]
 for $s_2=s_1rs\in S$ and $r\in R$.
 
 Consequently, $P\langle s_2\rangle\subseteq\bigcup_{i=1}^{n+1}P_i\langle s_2\rangle\subseteq\bigcup_{i\neq k}P_i$. Thus, by assumption, there exsits $s_3\in S$ such that $(P\langle s_2\rangle)\langle s_3\rangle\subseteq P_i$. Therefore, $P\langle s_4\rangle\subseteq P_i$ for some $s_4 =s_2r_1s_3 \in S$ and $r_1 \in R$. 
\end{proof}

\begin{proposition} Let S be an m-system of a ring R and P be an S-prime submodule of right R-module M. If N is  submodule of M with $S\cap(N:_RM)\neq \phi$ then $P(N:_RM)$ is an S-prime submodule.
\end{proposition}
\begin{proof}  Suppose $mRa\subseteq P(N:_RM)$ for some $a \in R$ and $m\in M$, then,  $mRa\subseteq P$, hence, either $m\langle s\rangle\subseteq P$ or $a\langle s\rangle\subseteq (P:_RM)$ for some $s\in S$. Let $s_1\in S\cap(N:_RM)$, then, 
\[
\text{either   }  m\langle s\rangle s_1\subseteq P(N:_RM) \text{  or  } Ma\langle s\rangle s_1\subseteq P(N:_RM). 
\]
Since $S$ is an $m$-system, there exists $r_1\in R$ such that $s_2=sr_1s_1\in S$. Thus, either $m\langle s_2\rangle \subseteq P(N:_RM)$ or $a\langle s_2\rangle \subseteq (P(N:_RM):_R M)$. Hence, $P(N:_RM)$ is an $S$-prime submodule.
\end{proof}

From the above proposition, we can conclude the following corollary.

\begin{corollary} Let S be an m-system of a ring R and P be an S-prime submodule of right R-module M. If I is  an ideal of R with   $S\cap I\neq \phi$ then $PI$ is an S-prime submodule.
\end{corollary}

\begin{theorem}\label{MODY} Let $f: M_1 \to M_2$ be an R-epimorphism. Consider $S \subseteq R$ as an $m$-system, and let $P$ be  an S-prime submodule of $M_1$ such that $\text{ker}(f) \subseteq P$. Then, $f(P)$ is an S-prime submodule of $M_2$.
\end{theorem}

\begin{proof} Let  $m_2Ra\subseteq f(P)$ for some $a \in R$ and $m_2\in M_2$, then, there exists $m_1\in M_1$ such that $f(m_1)=m_2$, hence, $f(m_1Ra)\subseteq f(P)$, consequently, $m_1Ra\subseteq f^{-1}(f(m_1Ra))\subseteq f^{-1}(f(P))=P$. Thus, either    $m_1\langle s\rangle\subseteq P$ or $a\langle s\rangle\subseteq (P:_RM_1)$ for some $s\in S$, and hence,  either $m_2\langle s\rangle\subseteq f(P)$ or $a\langle s\rangle\subseteq (f(P):_RM_2)$. 

Now we show that $(f(P):_R M_2)\cap S=\phi$. Assume $(f(P):_R M_2)\cap S\neq\phi$, then there exists $s\in(f(P):_R M_2)\cap S$, hence, $M_2s=f(M_1s)\subseteq f(P)$, and thus, $M_1s\subseteq f^{-1}(f(M_1s))\subseteq f^{-1}(f(P))=P$, consequently, $s\in (P:_RM_1)$, contradiction. Hence, $f(P)$ is an S-prime submodule of $M_2$.
\end{proof}

\begin{theorem}\label{K} Let $f$: $M_1 \to M_2$ be an R-homomorphism. Consider $S \subseteq R$ as an $m$-system, and let $P$ be  an S-prime submodule of $M_2$. If $(f^{-1}(P):_R M_1)\cap S=\phi$, then $f^{-1}(P)$ is an S-prime submodule of $M_1$. 
\end{theorem}

\begin{proof}   Let  $m_1Ra\subseteq f^{-1}(P)$ for some $a \in R$ and $m_1\in M_1$, then, 
\[
f(m_1)Ra=f(m_1Ra)\subseteq f(f^{-1}(P))\subseteq P, 
\]
hence, either    $f(m_1\langle s\rangle)=f(m_1)\langle s\rangle\subseteq P$ or $a\langle s\rangle\subseteq (P:_RM_2)$ for some $s\in S$, and hence, either $m_1\langle s\rangle\subseteq f^{-1}(P)$ or $a\langle s\rangle\subseteq (f^{-1}(P):_RM_1)$. Thus, $f^{-1}(P)$ is an S-prime submodule of $M_1$.   
\end{proof}

\begin{theorem} Let $S \subseteq R$ be an $m$-system, and let $P$ be  an S-prime submodule of right R-module M. If L is a submodule of M with $L\subseteq P$, then, $P\in Spec_S (M_R)$ if and only if $P/L\in Spec_S ((M/L)_R)$.
\end{theorem}

\begin{proof} Suppose $P$ is  an $S$-prime submodule of $M$, and let $f$: $M \to M/L$ be an $R$-homomorphism defined as $f(m)=m+L$ for $m\in M$. Then by Theorem \ref{MODY},  $f(P)=P/L$ is an $S$-prime submodule.

Conversely, suppose $P/L$ is  an $S$-prime submodule of $M/L$, and let $mRa\subseteq P$ for some $a\in R$ and $m\in M$, then, $(m+L)Ra\subseteq P/L$. Thus, either $(m+L)\langle s\rangle\subseteq P/L$ or $a\langle s\rangle\subseteq (P/L:_R M/L)$, which implise either $m\langle s\rangle\subseteq P$ or $a\langle s\rangle\subseteq (P:_R M)=(P/L:_R M/L)$. Hence, $P$ is an $S$-prime submodule.
\end{proof}

\begin{corollary} Let $P$ be  an S-prime submodule of a right R-module M, and let N be a submodule of M such that $(P :_R N) \cap S=\phi$. Then, $P\cap N$ is an S-prime submodule.
\end{corollary}

\begin{proof}  Let $f$: $N \to M$ be an $R$-homomorphism defined as $f(n)=n$ for $n\in N$. Since $f^{-1}(P)=P\cap N$, then, $(f^{-1}(P):_R N)\cap S=\phi$. Hence, by Theorem \ref{K}, $P\cap N$ is an $S$-prime submodule.
\end{proof}

We need the followings lemmas 

\begin{lemma}\label{A} Let $S$ be an $m$-system of a ring $R$. If $M$ is an $m$-system of a ring $T$, then, $S\times M$ is an $m$-system of the ring $R\times T$.

\end{lemma}
\begin{proof}  The proof is  routine.
\end{proof}

\begin{lemma}\label{MO} Let $R =R_1\times R_2$, where $R_1$ and $R_2$ are rings, and let $S= S_1\times S_2$, where $S_1$ and $S_2$ are  m-systems of $R_1$ and $R_2$, respectively. If $P= P_1\times P_2$ is an ideal of R, disjoint from S,  the following are equivalent: 

$(1)$ $P\in Spec_S(R)$. 

$(2)$ Either [$P_1 \in Spec_{S_1}( R_1)$ and $P_2\cap S_2\neq \phi$] or [$P_2\in Spec_{S_2}( R_2)$ and $P_1\cap S_1\neq\phi$].

\end{lemma}

\begin{proof} $(1)\Rightarrow(2)$ Suppose that $P$ is an $S$-prime ideal of $R$. 

Since 
$(a, 0)R(0, b)\subseteq P$ for all $a\in R_1$ and all $b \in R_2$, then, either $ (a, 0)\langle (s_1, s_2)\rangle\subseteq P$ or $ (0, b)\langle (s_1, s_2)\rangle\subseteq P$ for some $(s_1, s_2)\in S$, hence, either $ a\langle s_1\rangle\subseteq P_1$ or $ b\langle  s_2\rangle\subseteq P_2$. For $a=1$ and $b=1$, we obtain  either $ s_1\in\langle s_1\rangle\subseteq P_1$ or $ s_2\in\langle  s_2\rangle\subseteq P_2$. Thus, either  $P_1\cap S_1\neq\phi$ or $P_2\cap S_2\neq \phi$. 

Now, without loss of generality, suppose that $P_2\cap S_2\neq \phi$, then we should show that $P_1 \in Spec_{S_1}( R_1)$. Let $IJ\subseteq P_1$ for some ideals $I$, $J$ of $R_1$, then, $(I,0)\times(J,0)\subseteq P_1\times P_2$, and hence, either $(I,0)\langle (s_1, s_2)\rangle\subseteq P_1\times P_2$ or $(J,0)\langle (s_1, s_2)\rangle\subseteq P_1\times P_2$, consequently, either $I\langle s_1\rangle\subseteq P_1$ or $J\langle s_1\rangle\subseteq P_1$. Since $P$ is disjoint from $S$, we get $P_1 \cap S_1 =\phi$. Thus, $P_1$ is an $S_1$-prime ideal of $R_1$. If $P_1\cap S_1\neq\phi$, then similar to the previous discussion we can prove $P_2$ is an $S_2$-prime ideal of $R_2$.

$(2)\Rightarrow(1)$ Without loss of generality, suppose that $P_1 \in Spec_{S_1}( R_1)$ and $P_2\cap S_2\neq \phi$.  Let $(I_1\times J_1)(I_2\times J_2)\subseteq P$ for some ideals $I_1$, $I_2$ of $R_1$ and some ideals $J_1$, $J_2$ of $R_2$, then, $(I_1 I_2)\times(J_1 J_2)\subseteq P_1\times P_2$, hence,  $I_1 I_2\subseteq P_1$, consequently, either $I_1\langle s_1\rangle\subseteq P_1$ or $I_2\langle s_1\rangle\subseteq P_1$ for some $s_1 \in S_1$. On the ather hand, let $s_2\in P_2\cap S_2$, in this case we obtain  

\[
\text{either  } (I_1\langle s_1\rangle, J_1\langle s_2\rangle)\subseteq P_1\times P_2 \text{  or  } (I_2\langle s_1\rangle, J_2\langle s_2\rangle)\subseteq P_1\times P_2.
\]

Hence, 
\[
\text{either  } (I_1\times J_1)\langle (s_1,s_2)\rangle\subseteq P_1\times P_2 \text{  or  } (I_2\times J_2)\langle (s_1,s_2)\rangle\subseteq P_1\times P_2.
\]
Thus, $P$ is an $S$-prime ideal of $R$. 

If $P_2 \in Spec_{S_2}( R_2)$ and $P_1\cap S_1\neq \phi$, then, similar to the previous discussion we can see that $P$ is an $S$-prime ideal of $R$. 
\end{proof}

\begin{theorem} Let  $R_1$ and $R_2$ be any rings, and let $M =M_1\times M_2$ be a right $R$-module, where $R=R_1\times R_2$  and let $S= S_1\times S_2$, where $S_1$ and $S_2$ are  m-systems of $R_1$ and $R_2$, respectively. If  $P= P_1\times P_2$ is a submodule of M, then, the followings are equivalent: 

$(1)$ $P\in Spec_S (M_R)$. 

$(2)$ Either [$P_1 \in Spec_{S_1}( {M_1}_{R_1})$ and $(P_2:_{R_2} M_2)\cap S_2\neq \phi$] or [$P_2 \in Spec_{S_2}( {M_2}_{R_2})$ and $(P_1:_{R_1} M_1)\cap S_1\neq \phi$].
\end{theorem}

\begin{proof} $(1)\Rightarrow(2)$ Suppose $P$ is an $S$-prime submodule  of $M$, then, by Proposition \ref{M1}, $(P:_R M)$ is an $S$-prime ideal of $R$. However, 
\[
(P:_R M)=(P_1:_{R_1} M_1)\times(P_2:_{R_2} M_2).
\]
Thus, by Lemma \ref{MO}, either   
\[
 (P_1:_{R_1} M_1) \in Spec_{S_1}( R_1) \text{   and   } (P_2:_{R_2} M_2)\cap S_2\neq \phi,
 \]
or 
  \[
  (P_2:_{R_2} M_2) \in Spec_{S_2}(R_2) \text{   and   }(P_1:_{R_1} M_1)\cap S_1\neq \phi. 
  \]
  Without loss of generality, suppose that $(P_2:_{R_2} M_2)\cap S_2\neq \phi$, then, we should show that $P_1 \in Spec_{S_1}({M_1}_{R_1})$. Let $NI\subseteq P_1$ for some ideal $I$ of $R_1$, and some submodule $N$, of $M_1$. Then, 
  \[
   (NI \times 0_{M_2})=(N,0_{M_2})\times(I,0_{R_2})\subseteq P_1\times P_2=P,
\]
 hence,
 \[
  \text{either    } (N,0_{M_2})\times\langle (s_1,s_2)\rangle\subseteq P_1\times P_2 \text{   or   } (I,0_{R_2})\times\langle (s_1,s_2)\rangle\subseteq (P:_R M),
\]
 for some $(s_1, s_2)\in S$, hence,
\[
  \text{either    } N\langle s_1\rangle\subseteq P_1 \text{   or   } I\langle s_1\rangle\subseteq (P_1:_{R_1} M_1).
\]
Thus, $P_1$ is an $S_1$-prime submodule of $M_1$. If $(P_1:_{R_1} M_1)\cap S_1\neq \phi$. , then similar to the previous discussion we can prove $P_2$ is an $S_2$-prime submodule of $M_2$.

$(2)\Rightarrow(1)$   Without loss of generality, suppose that $P_1 \in Spec_{S_1}( {M_1}_{R_1})$ and $(P_2:_{R_2} M_2)\cap S_2\neq \phi$.  Let $(N_1\times N_2)(I_1\times I_2)\subseteq P$ for some ideals $I_1$, $I_2$ of $R$ and some submodules $N_1$, $N_2$ of $M$, then, $(N_1 I_1)\times(N_2 I_2)\subseteq P_1\times P_2$, hence,  $N_1 I_1\subseteq P_1$, consequently, either $N_1\langle s_1\rangle\subseteq P_1$ or $I_1\langle s_1\rangle\subseteq (P_1:_{R_1} M_1)$ for some $s_1 \in S_1$. On the other hand, let $s_2\in (P_2:_{R_2} M_2)\cap S_2$, in this case we obtain  $N_2\langle s_2\rangle\subseteq M_2\langle s_2\rangle\ \subseteq P_2$ and $I_2\langle s_2\rangle\subseteq(P_2:_{R_2} M_2)$. Thus, for $s=(s_1,s_2)$, we have either

\[
	 (I_1\times I_2)\langle s\rangle=(I_1\langle s_1\rangle)\times(I_2\langle s_2\rangle)\subseteq (P_1:_{R_1} M_1)\times(P_2:_{R_2} M_2)=(P:_R M), 
\]

or

\[
 (N_1\times N_2)\langle s\rangle=(N_1\langle s_1\rangle)\times(N_2\langle s_2\rangle)\subseteq P_1\times P_2=P.
 \]
 
 Therefore, $P$ is an $S$-prime submodule of $M$. 

If $P_2 \in Spec_{S_2}( {M_2}_{R_2})$ and $(P_1:_{R_1} M_1)\cap S_1\neq \phi$, then, similar to the previous discussion we can see that $P$ is an $S$-prime submodule of $M$. 
\end{proof}

\section{Right $S$-Noetherian Rings}

\begin{definition} Let $S$ be an $m$-system of a ring $R$, and $M$ a right $R$-module.

$(1)$ A submodule $N$ of $M$ is called $S$-finite ($S$-principal) if there exist an element $s \in S$ and a finitely
generated (principal) submodule $F$ of $M$ such that $N\langle s\rangle \subseteq F\subseteq N$. Particularly, $M$ is called $S$-finite ($S$-principal) if there exists a finitely generated (principal) submodule $F$ of $M$ such that $M\langle s\rangle \subseteq F$.

$(2)$ The right $R$-module $M$ is called $S$-Noetherian if every submodule of $M$ is $S$-finite.

$(3)$ [Definition 3.6 of \cite{A}] 
Let $S$ be an $m$-system of a ring $R$. A ring $R$ is called right  $S$-Noetherian if every right  ideal of $R$ is   $S$-finite as a right $R$-module.
\end{definition}

 Let $S$ be an $m$-system of a ring $R$, and $M$ a right $R$-module. Every submodule of an $S$-Noetherian module is $S$-Noetherian. In addition, if $S1 \subseteq S2$ are $m$-systems of a ring, then any $S_1$-Noetherian module is $S_2$-Noetherian.

 \begin{proposition}  Let $S$ be an $m$-system of a ring $R$, and $M$ a multiplication finitely generated right $R$-module. If every   S-prime submodule is S-finite, then every prime submodule  is S-finite.  
  \end{proposition}

 \begin{proof}  Let $N$ be a prime submodule of $M$. If $(N:_RM)\cap S= \phi$, then, by Example \ref{1}, $N$ is an $S$-prime submodule, hence $N$ is $S$-finite.  If $(N:_RM)\cap S\neq \phi$, then for $s\in (N:_RM)\cap S$, $(N:_RM)\langle s\rangle \subseteq\langle s\rangle \subseteq(N:_RM)$, hence,
 
 \[  
 N\langle s\rangle=M(N:_RM)\langle s\rangle \subseteq M\langle s\rangle \subseteq M(N:_RM)=N.
 \] 
 Since $M\langle s\rangle$ is a finitely generated submodule, then $N$ is $S$-finite.
  \end{proof} 
 
 \begin{corollary}  Let $S$ be an $m$-system of a ring $R$, and $M$ a multiplication finitely generated right $R$-module. If $(N:_RM)$  is an S-prime ideal of R, then M is an S-Noetherian R-module.
  \end{corollary}

 \begin{proof}  Let $N$ be a  submodule of $M$, then $(N:_RM)$ is $S$-finite, thus, there exist $s\in S$ and a finitely generated right ideal $I$ such that
 
 \[
 (N:_RM)\langle s\rangle \subseteq I \subseteq (N:_RM).
\]
 Hence, 
 \[  
 N\langle s\rangle=M(N:_RM)\langle s\rangle \subseteq MI \subseteq M(N:_RM)=N.
 \] 
Since $MI$ is a finitely generated submodule, then, $N$ is $S$-finite, and hence, $M$ is $S$-Noetherian.
  \end{proof}

Recall that since every multiplicatively closed set is an $m$-system, then every  right $S$-Noetherian ring associated with multiplicatively closed set $S$, is an $S$-Noetherian ring associated with the $m$-system $S$.

In the following we give some examples of right  $S$-Noetherian rings.

\begin{example}\label{WOOD} 	Let $R=\left[ \begin{array}{cc}\mathbb{Q}& \mathbb{Q} \\ 0 & \mathbb{Z}\end{array} \right]$, since Z is a Noetherian domain which does not equal to its fractions field, then by Corollary 1.23 of \cite{L}, R is not right Noetherian. However, by considering the $m$-system 

$S=\bigg\{\left[ \begin{array}{cc}2^n& 0 \\ 0 & 0\end{array} \right], n\in \mathbb{N}\cup\{0\}\bigg\}$, and $s=\left[ \begin{array}{cc}1& 0 \\ 0 & 0\end{array} \right]$, any right ideal $I$ of $R$ is $S$-finite, because for the right ideal $I$, by Proposition 1.17 of \cite{L}, we have the followings forms:

If $I=\left[\begin{array}{cc}0& 0 \\ 0 & J\end{array} \right]$, for some ideal $J$ of $\mathbb{Z}$, then, 
 $I\langle s \rangle\subseteq\left[ \begin{array}{cc}0& 0 \\ 0 & 0\end{array} \right]\subseteq I$. 
 
 If $I=\left[ \begin{array}{cc}0& K \\ 0 & J\end{array} \right]$, where $J$ is an ideal of $\mathbb{Z}$, and  $K$ is a right $Z$-submodule of the right $Z$-module Q, then, 
 $I\langle s\rangle\subseteq\left[ \begin{array}{cc}0& 0 \\ 0 & 0\end{array} \right]\subseteq I$.

 If $I=\left[ \begin{array}{cc}\mathbb{Q}& \mathbb{Q} \\ 0 & J\end{array} \right]$, where $J$ is an ideal of $\mathbb{Z}$, then, 
 \[
 I\langle s\rangle\subseteq\left[ \begin{array}{cc}\mathbb{Q}& 0 \\ 0 & 0\end{array} \right]\subseteq\left[ \begin{array}{cc}\mathbb{Q}& \mathbb{Q} \\ 0 & 0\end{array} \right]=\left[ \begin{array}{cc}1& 0 \\ 0 & 0\end{array} \right]R\subseteq I. 
 \]

 Thus, $R$ is  right $S$-Noetherian.
\end{example}

\begin{example} 	Let $R=\left[ \begin{array}{cc}\mathbb{Q}& \mathbb{Q} \\ 0 & \mathbb{Z}\end{array} \right]$. Then, consider the ring

\[
 T_2(R)=\left[\begin{array}{cc}R& R \\ 0 & R\end{array} \right].
 \]

Take
 \[
 S_{T_2(R)}=\Bigg\{\left[ \begin{array}{cc}s& 0\\ 0 & s\end{array} \right]; s\in S \Bigg\}, 
 \]
 where $S$  is the $m$-system defined in Example \ref{WOOD}.
 Then, $S_{T_2(R)}$ is an $m$-system of the ring $T_2(R)$ ($T_2(R)$ is the upper triangular matrices ring over $R$).  The ring $T_2(R)$ is a right $S_{T_2(R)}$-Noetherian, by Theorem 3.9 of \cite{A},
\end{example}

\subsection*{Conflict of interest}

The author declares no conflicts of interest.

\subsection*{Acknowledgment}
I would like to extend my thanks to professor {\c S}ehmus F{\i}nd{\i}k for his useful comments, I also extend my gratitude to professor Jongwook Baeck for his help in Example \ref{WOOD}.

\end{document}